\documentclass[12pt,a4paper]{amsart}

\usepackage{amssymb} 
\usepackage{amsmath} 
\usepackage{amscd}
\usepackage{amsbsy}
\usepackage[matrix,arrow]{xy}
\usepackage{comment}

\DeclareMathOperator{\GL}{GL}

\DeclareMathOperator{\Sel}{Sel}

\DeclareMathOperator{\Img}{Im}
\DeclareMathOperator{\Res}{Res}

\newcommand{\Q}{{\mathbb Q}}
\newcommand{\Z}{{\mathbb Z}}
\newcommand{\F}{{\mathbb F}}
\newcommand{\PP}{{\mathbb P}}
\newcommand{\cA}{\mathcal{A}}
\newcommand{\cB}{\mathcal{B}}

\newcommand{\cS}{{\mathcal{S}}}

\newcommand{\OO}{{\mathcal O}}

\newcommand{\fp}{\mathfrak{p}}

\newcommand{\Qbar}{\bar{\Q}}
\newcommand{\GalQ}{{\mathcal G}}
\newcommand{\Map}{\operatorname{Map}}
\newcommand{\To}{\longrightarrow}
\newcommand{\fake}{{\text{\rm fake}}}
\newcommand{\mult}{^\times}
\newcommand{\id}{\operatorname{id}}
\newcommand{\pr}{\operatorname{pr}}

\newcounter{nootje}
\newtheorem{thm}{Theorem}[section]
\newtheorem{lem}[thm]{Lemma}

\theoremstyle{definition}

\theoremstyle{remark}

\begin{document}

\title[Partial Descent on Hyperelliptic Curves]%
      {Partial Descent on Hyperelliptic Curves \\
       and the Generalized Fermat Equation $x^3+y^4+z^5=0$}
\author{Samir Siksek}
\address{Mathematics Institute\\
	University of Warwick\\
	Coventry\\
	CV4 7AL \\
	United Kingdom}

\email{s.siksek@warwick.ac.uk}

\author{Michael Stoll}
\address{Mathematisches Institut,
         Universit\"at Bayreuth,
         95440 Bayreuth, Germany.}
\email{Michael.Stoll@uni-bayreuth.de}

\date{\today}
\thanks{The first-named 
author is supported by an EPSRC Leadership Fellowship.}


\begin {abstract}
Let $C : y^2=f(x)$ be a hyperelliptic curve defined over $\Q$.
Let $K$ be a number field and suppose $f$ factors over $K$
as a product of irreducible polynomials $f=f_1 f_2 \dots f_r$.
We shall define a \lq\lq Selmer set\rq\rq\ 
corresponding to this factorization
with the property that if it is empty then $C(\Q)=\emptyset$.
We shall demonstrate the effectiveness of our new method
by solving the generalized Fermat equation with signature $(3,4,5)$,
which is unassailable via the previously existing methods.
\end {abstract}
\maketitle

%======================================================================

\section{Introduction}

Let $f$ be a separable polynomial with coefficients in $\Z$ and 
degree $d \geq 3$. Let
$C$ be the non-singular projective hyperelliptic
curve with affine patch
\[
C : y^2=f(x).
\] 
One is interested in studying the set of rational points on $C(\Q)$
and, in particular, deciding whether $C(\Q)$ is empty or not. 
Several techniques have been developed to attack this problem 
\cite{BF}, \cite{BS},
\cite{BS2}, \cite{BS3}.
The easiest general method is the \lq\lq two-cover descent\rq\rq\
of Bruin and Stoll \cite{BS2}. Let 
$L$ be the \'{e}tale algebra $L=\Q[x]/f$;
this algebra is the direct sum of number fields. Bruin and Stoll
define a map from $C(\Q)$ to a group $H$ which is 
either $L\mult/(L\mult)^2$ or
$L\mult/(\Q\mult (L\mult)^2)$ depending on the parity of~$d$.\footnote{%
$R\mult$ denotes the multiplicative group of an algebra~$R$, and $(R\mult)^2$ denotes
the subgroup of squares.}
They show how to compute a finite subset of $H$,
which they call the \lq\lq fake 2-Selmer set\rq\rq, that contains
the image of $C(\Q)$ in $H$. If this fake 2-Selmer set is empty
then so is $C(\Q)$. 

The computation of the fake Selmer set requires  
explicit knowledge of the class and unit groups
of the number fields that are the direct summands of $L$.
If these number fields have large degrees or discriminants,
then this computation is impractical.
In this paper, we look at the situation where there is
a number field $K$ of relatively small degree such that
$f$ factors over $K$ into a product $f=f_1 f_2 \cdots f_r$
of irreducible factors. We define an appropriate 
\lq\lq Selmer set\rq\rq\ whose computation demands
the knowledge of the class group and units of $K$ but
not of larger number fields. This can be helpful either in
proving the non-existence of rational points on our
hyperelliptic curve, or in the construction of unramified
covers that can help in the determination of the set
of rational points. We call our method \lq\lq partial descent\rq\rq\
as the information it yields is usually weaker than the information
given by the fake 2-Selmer set. We explain this in more detail
in Section~\ref{S:pdesc} below

We shall demonstrate the effectiveness of our new method
by solving the generalized Fermat equation with signature $(3,4,5)$.

%\subsection{The Generalized Fermat Equation}

Let $p$, $q$, $r \in \Z_{\geq 2}$. The equation
\begin{equation}\label{eqn:FCgen}
x^p+y^q=z^r
\end{equation}
is known as the \emph{generalized Fermat equation}
(or the \emph{Fermat-Catalan} equation) with signature $(p,q,r)$.
As in Fermat's Last Theorem, one is interested in non-trivial
primitive integer solutions.
An integer solution $(x,y,z)$ is said to be {\em non-trivial} if
$xyz \neq 0$, and {\em primitive} if $x$, $y$, $z$ are coprime.
Let $\chi=p^{-1}+q^{-1}+r^{-1}$. The parametrization
of non-trivial primitive integer solutions for $(p,q,r)$ with
$\chi \geq 1$ has now been completed (\cite{Be}, \cite{Ed}).
The generalized Fermat Conjecture \cite{Da97}, \cite{DG}
is concerned with the case $\chi<1$.
It states that---up to sign and permutation---the only non-trivial 
primitive solutions to
\eqref{eqn:FCgen} with $\chi<1$ are 
\begin{gather*}
1+2^3 = 3^2, \quad 2^5+7^2 = 3^4, \quad 7^3+13^2 = 2^9, \quad
2^7+17^3 = 71^2, \\
3^5+11^4 = 122^2, \quad 17^7+76271^3 = 21063928^2, \quad
1414^3+2213459^2 = 65^7, \\
9262^3+15312283^2 = 113^7, \quad
43^8+96222^3 = 30042907^2, \quad 33^8+1549034^2 = 15613^3.
\end{gather*}
The generalized Fermat Conjecture
has been established for many signatures $(p,q,r)$,
including for several infinite families of signatures:
Fermat's Last Theorem $(p,p,p)$ by
Wiles and Taylor \cite{W}, \cite{TW};
$(p,p,2)$ and $(p,p,3)$ by Darmon and Merel \cite{DM};
$(2,4,p)$ by Ellenberg \cite{El} and Bennett, Ellenberg and Ng
\cite{BEN};
$(2p,2p,5)$ by Bennett \cite{Ben}.
Recently, Chen and Siksek \cite{ChenS} 
have solved the generalized Fermat equation
with signatures $(3,3,p)$ for a set of prime exponents $p$ having 
Dirichlet density $28219/44928$.
For exhaustive surveys see \cite[Chapter 14]{Cohen} and \cite{Be}. An older but
still very useful survey is \cite{Kr99}.
An up-to-date list of partial results with references is found in \cite{PSS}.
It appears that the `smallest' signature $(p,q,r)$ for which the
equation has not yet been resolved is $(3,4,5)$, in the sense that it is
the 
only 
signature with $\max\{p,q,r\} \le 5$ that is still open.
In this paper we shall prove the following theorem.
\begin{thm}\label{thm:main}
The only primitive integer solutions to the equation
\begin{equation}\label{eqn:main}
x^3+y^4+z^5=0
\end{equation}
are the trivial solutions $(0,\pm 1, -1)$, $(-1,0,1)$, $(1,0,-1)$,
$(-1,\pm 1,0)$.
\end{thm}
Our proof proceeds as follows. Edwards parametrized the primitive
solutions to the generalized Fermat equation $x^2+y^3+z^5=0$.
Using this parametrization, we reduce the resolution of \eqref{eqn:main}
to the determination of the set of rational points on $49$ 
hyperelliptic curves $C_i : y^2=f_i(x)$
where the polynomial $f_i$ has degree $30$ (and so the curves
are of genus~$14$). 
Of these $49$ we can eliminate $26$ using local considerations,
which leaves $23$ curves $C_i$.
These include $13$ \lq difficult\rq\
curves where the $f_i$ are irreducible.
An application of the method of Bruin and Stoll would require the computation of the
class and unit groups of number fields of degree $30$, which is
impractical at present. However, in these
$13$ cases there are number fields $K_i$ of degree $5$ such
that $f_i$ factors over~$K_i$; we shall 
show that an appropriate Selmer set corresponding to this
factorization is empty, and deduce that $C_i(\Q)=\emptyset$.
For the ten remaining cases, our Selmer sets are non-empty, but
we use them to construct unramified covers of the $C_i$.
It turns out that these unramified covers have low genus quotients
for which it is easy to determine the rational points. In this way
we can first determine the set of rational points on each of
the unramified covers and then on the remaining curves~$C_i$.

We are grateful to Don Zagier for suggesting to us that
\eqref{eqn:main} is the \lq next case\rq\ of the generalized
Fermat conjecture. The first-named author would like
to thank John Cremona, Sander Dahmen 
and Michael Mourao for helpful discussions. 

%======================================================================

\section{Partial Descent on Hyperelliptic Curves} \label{S:pdesc}

We will use $\GalQ$ to denote the absolute Galois group of~$\Q$.

It is convenient to adopt homogeneous coordinates for hyperelliptic curves.
Let $f(u,v)$ be a squarefree binary form of even degree $2d \ge 4$
with rational integer coefficients. Let
\begin{equation}\label{eqn:Cmodel}
  C : y^2 = f(u,v)
\end{equation}
be the hyperelliptic curve of genus $g = d-1$
associated to~$f$ in weighted projective space~$\PP_{(1,1,d)}$,
where $u$, $v$, $y$ are respectively given weights
$1$, $1$ and~$d$. Then $C$ is covered by the two affine curves
$y^2 = f(x,1)$ and $y^2 = f(1,x)$. To avoid having to deal with
special cases, we will assume that $f(1,0) \neq 0$ (so that
$f(x,1)$ has degree~$2d$). If $f(1,0) = 0$, then we have to work
with an extra factor~$v$ in the factorisation of~$f$ below, and
everything goes through in much the same way.

Over~$\Qbar$, we can write
\[ f(u,v) = c (u - \theta_1 v) \cdots (u - \theta_{2d} v) \]
where $c \in \Q\mult$ is the leading coefficient of~$f(x,1)$ and the
$\theta_j \in \Qbar$ are the (pairwise distinct) roots of~$f(x,1)$.
Let $\Theta = \{\theta_1, \ldots, \theta_{2d}\}$; this is a set
on which the absolute Galois group of~$\Q$ acts. Denote by
\[ L = \Map_\Q(\Theta, \Qbar) \cong \Q[T]/(f(T,1)) \]
the corresponding \'etale algebra (its elements are Galois-equivariant
maps from $\Theta$ to~$\Qbar$). This algebra~$L$ decomposes as a
product of number fields corresponding to the Galois orbits on~$\Theta$,
or equivalently, to the irreducible factors of~$f$ in~$\Q[u,v]$.
Given elements $\alpha \in L\mult$ and $s \in \Q\mult$
such that $c N_{L/\Q}(\alpha) = s^2$, we can define
a curve $D_{\alpha,s} \subset \PP^{2d-1} \times C$ by declaring that
\begin{align*}
  \bigl((z_1 &: \ldots : z_{2d}), (u : v : y)\bigr) \in D_{\alpha,s} \\
     &\iff \exists a \neq 0 \text{\ such that } \forall 1 \le j \le 2d :
            \alpha(T_j) z_j^2 = a (u - \theta_j v)
            \text{\ and\ } s z_1 \cdots z_{2d} = a^d y \,.
\end{align*}
(We consider $\alpha \in L$ as a map $\alpha : \Theta \to \Qbar$.)
It is clear that the condition is invariant under scaling and
under the action of the Galois group~$\GalQ$,
so that $D_{\alpha,s}$   
is defined over~$\Q$. Note that the hyperelliptic
involution on~$C$ induces an isomorphism between $D_{\alpha,s}$
and~$D_{\alpha,-s}$, and both curves are isomorphic to their common
projection~$D_\alpha$ to~$\PP^{2d-1}$. Projection to the second factor
induces a covering map $\pi_{\alpha,s} : D_{\alpha,s} \to C$.
It can be checked that this map is an unramified covering of~$C$
of degree~$2^{2d-2} = 2^{2g}$;
more precisely, $\pi_{\alpha,s}$ is a $C$-torsor under~$J[2]$, the
2-torsion subgroup of the Jacobian variety~$J$ of~$C$: a \emph{2-covering}
of~$C$. Two such
2-coverings $\pi_{\alpha,s}$ and~$\pi_{\beta,t}$ are isomorphic over~$\Q$
as coverings of~$C$ if and only if there are $\gamma \in L\mult$ and
$w \in \Q\mult$
such that $\beta = \alpha \gamma^2 w$ and $t = s N_{L/\Q}(\gamma) w^d$.
The set of isomorphism classes of 2-coverings of~$C$ that have points
everywhere locally is called the \emph{2-Selmer set}
$\Sel^{(2)}(C/\Q)$ of~$C$. Since
it can be shown that every such 2-covering can be realized in the
form~$\pi_{\alpha,s}$, it follows that the 2-Selmer set can be identified with
a subset of
\[ H_c = \frac{\{(\alpha, s) \in L\mult \times \Q\mult
                  : c N_{L/\Q}(\alpha) = s^2\}}%
              {\{(\gamma^2 w, N_{L/\Q}(\gamma) w^d)
                  : \gamma \in L\mult, w \in \Q\mult\}} \,.
\]
(The group below acts on the set above
by multiplication; the quotient is with respect to this group action.)
It is known that the 2-Selmer set is finite.
The image of~$\Sel^{(2)}(C/\Q)$ under the map to $L\mult/(\Q\mult (L\mult)^2)$
is known as the \emph{fake 2-Selmer set} $\Sel^{(2)}_\fake(C/\Q)$ of~$C$.
The map $\Sel^{(2)}(C/\Q) \to \Sel^{(2)}_\fake(C/\Q)$ is either a
bijection or two-to-one. There is a natural map $C(\Q) \to \Sel^{(2)}(C/\Q)$
given on points with nonvanishing $y$-coordinate by
\[ \delta : C(\Q) \ni (u_0 : v_0 : y_0)
                \longmapsto [u_0 - T v_0, y_0] \in H_c \]
(where the square brackets denote the element represented by $(u_0 - Tv_0, y_0)$ and
$T \in L$ is the identity map). If $y_0 = 0$, then we have $v_0 \neq 0$, and
we can write
\[ f(u,v) = c (u - \theta_0 v) \tilde{f}(u,v)
       \quad\text{with \quad $\tilde{f}(x,1)$ monic,}
\]
where $\theta_0 = u_0/v_0$. Then we set (compare \cite{Sch} and~\cite{PS})
\[ \delta\bigl((u_0 : v_0 : 0)\bigr)
      = [\theta_0-T + c\tilde{f}(T,1), c\tilde{f}(\theta_0 ,1)]. 
\]
Therefore if the 2-Selmer set or the fake 2-Selmer set of~$C$ is empty, then
$C$ cannot have any rational points. The map~$\delta$ above has the property
that if the image of~$P \in C(\Q)$ under~$\delta$ corresponds to $\pi_{\alpha,s}$,
then $P = \pi_{\alpha,s}(Q)$ for some $Q \in D_\alpha(\Q)$.
Denoting the analogues of $H_c$ and~$\delta$ over~$\Q_p$ by $H_{c,p}$
and~$\delta_p$, and writing $\rho_p : H_c \to H_{c,p}$ for the canonical map,
we have that
\[\Sel^{(2)}(C/\Q)
   = \{h \in H_c : \rho_p(h) \in \Img{\delta_p} \text{\ for all places $p$ of~$\Q$}\}
   \,.
\]

For a detailed account of the theory of
2-descent on hyperelliptic curves, see~\cite{BS2}.
There it is shown how the fake 2-Selmer set can be computed
if one can determine the class group of~$L$ (i.e., the class groups
of the various number fields occurring as factors of~$L$) and an
odd-index subgroup of the group of units of~$L$ (dito). Now if the
irreducible factors of~$f$ have large degree, then this information
may be hard or next to impossible to get. So we would like to be able
to compute some kind of intermediate Selmer set with less effort,
at the price of potentially obtaining less information. This is the
`partial 2-descent' that we now describe.

We first define a different model of~$D_{\alpha,s}$ that includes a lot
of redundant variables. We denote by~$\Pi$ the set of all subsets of
the set~$\Theta = \{\theta_1, \ldots, \theta_{2d}\}$ of roots of~$f$.
We will frequently identify elements $\varphi$ of~$\Pi$ with the
corresponding factor $\prod_{\theta \in \varphi} (u - \theta v)$ of~$f$.
For any $\GalQ$-invariant subset $Y$ of~$\Pi$,
we let $\PP_Y$ be the weighted projective
space over~$\Q$ whose coordinates correspond to the elements of~$Y$,
with weights given by their cardinality (as subsets of~$\Theta$)
or degrees (as factors of~$f$), and with twisted Galois action given
by the $\GalQ$-action on~$Y$. We write $z_\varphi$ for the coordinate
corresponding to~$\varphi \in Y$. Then we can define an embedding
\[ \jmath : \PP_\Theta \To \PP_\Pi \,, \quad
            (z_1 : \ldots : z_{2d}) \longmapsto
              ( \ldots : \prod_{i : \theta_i \in \varphi} z_i : \ldots)
\]
where the product gives the $z_\varphi$-coordinate of the image point.
Now we let
\[ \tilde{D}_{\alpha, s}
     = (\jmath \times \id_C)(D_{\alpha, s}) \subset \PP_\Pi \times C \,.
\]

For each $\GalQ$-invariant subset $Y \subset \Pi$, there is a projection
$\pi_Y : \PP_\Pi \to \PP_Y$. We obtain a commutative diagram
\[ \xymatrix{ \tilde{D}_{\alpha,s} \ar[rr]^{\tilde{\pi}_{\alpha,s} }
                                   \ar[dr]_{\pi_Y \times \id}
                & & C \\
                & D_{\alpha,s,Y} \ar[ur]_{\pi_{\alpha,s,Y}}
            }
\]
We write $\pi_{\alpha,s,Y}$ for the restriction of the second projection
$\pr_2 : \PP_Y \times C \to C$ to the image~$D_{\alpha,s,Y}$ of~$\tilde{D}_{\alpha,s}$
under~$\pi_Y \times \id_C$. With these notations,
$D_{\alpha,s} = D_{\alpha,s,\Theta}$ and $\pi_{\alpha,s} = \pi_{\alpha,s,\Theta}$,
and similarly $\tilde{D}_{\alpha,s} = D_{\alpha,s,\Pi}$,
$\tilde{\pi}_{\alpha,s} = D_{\alpha,s,\Pi}$.

It is well-known that
\[ J[2] \cong \frac{\ker(N : \mu_2^\Theta \to \mu_2)}{\mu_2}
        = \frac{\ker(N_{\bar{L}/\Qbar} : \mu_2(\bar{L}) \to \mu_2)}{\mu_2}
\]
where $\bar{L} = L \otimes_{\Q} \Qbar$ and $N$ maps an element
of~$\mu_2^\Theta$ to the product of its entries. The elements
of~$\mu_2^\Theta$ correspond to the subsets of~$\Theta$ in a natural way by
\begin{equation} \label{PimuTh}
   \alpha \longmapsto \{\theta \in \Theta : \alpha(\theta) = -1\} \,.
\end{equation}
The 2-torsion points therefore correspond to the partitions of~$\Theta$
into two sets of even cardinality, and addition in~$J[2]$ corresponds
to taking symmetric differences. The action of~$J[2]$ on a covering~$D_{\alpha,s}$
is given in this setting by
\[ P + \bigl((z_1 : \ldots : z_{2d}), Q)
     = \bigl((\alpha(\theta_1) z_1 : \ldots : \alpha(\theta_{2d}) z_{2d}), Q\bigr)
\]
where $P \in J[2]$ is represented by $\alpha \in \mu_2^\Theta$.
Similarly, the action on~$D_{\alpha,s,Y}$ is given by multiplying $z_\varphi$
with $\prod_{\theta \in \varphi} \alpha(\theta)$, for all $\varphi \in Y$.

We use this description to find the group~$\Gamma_Y \subset J[2]$
of deck transformation of the
covering $\tilde{D}_{\alpha,s} \to D_{\alpha,s,Y}$. Its elements are represented
by those $\alpha \in \mu_2^\Theta$ with $N(\alpha) = 1$ for which there
is $\varepsilon \in \mu_2$ such that
$\prod_{\theta \in \varphi} \alpha(\theta) = \varepsilon^{\#\varphi}$
for all $\varphi \in Y$. Since we can replace $\alpha$ by $\alpha\varepsilon$,
we can take $\varepsilon = 1$. By Galois theory, it follows that
$\pi_{\alpha,s,Y} : D_{\alpha,s,Y} \to C$ is (geometrically) Galois
with Galois group $G_Y \cong J[2]/\Gamma_Y$. This group $G_Y$ is
dual to the annihilator of~$\Gamma_Y$ under the Weil pairing on~$J[2]$.
Recall that the Weil pairing is determined by the parity of the cardinality
of the intersections of the sets in the partitions corresponding to
two elements of~$J[2]$. If we identify $Y$ with a subset of~$\mu_2^\Theta$
via~\eqref{PimuTh}, then the dual group is the image of
$\langle Y \rangle \cap \ker N$ in $J[2] = (\ker N)/\mu_2$.

If $G_Y^\vee$ is neither trivial nor all of~$J[2]$, then we obtain
intermediate coverings. It should be noted that this is not possible
in the generic case when the Galois group of~$f$ is the full symmetric
group~$S_{2d}$, since then the minimal $\GalQ$-invariant subsets~$Y$ of~$\Pi$
contain all subsets of some fixed cardinality, and each such~$Y$ generates
all of~$J[2]$. However, in many cases of interest, there are additional
symmetries present that lead to smaller Galois groups, so that intermediate
coverings may be available.

We want to generalize our setting for Selmer sets. To this end, we
proceed similarly as above. We denote by~$L_Y$ the \'etale $\Q$-algebra
corresponding to the $\GalQ$-set~$Y$ (then $L = L_\Theta$). We have the subgroup
\[ U_\Pi = \{\alpha \in L\mult
             : \alpha(\varphi) = \prod_{\theta \in \varphi} \alpha(\theta)
               \text{\ for all $\varphi \in \Pi$}\}
         \subset L_\Pi\mult
\]
with an embedding
\[ \iota_\Pi : \Q\mult \To U_\Pi\,, \quad
               a \longmapsto (\varphi \mapsto a^{\#\varphi})\,,
\]
the set
\[ V_{c,\Pi} = \{(\alpha, s) \in U_\Pi \times \Q\mult : c \alpha(\Theta) = s^2\}
             \subset U_\Pi \times \Q\mult \,,
\]
on which the group
\[ W_\Pi = \{(\gamma^2 \iota_\Pi(a), a^d \gamma(\Theta))
                       : \gamma \in U_\Pi, a \in \Q\mult\}
         \subset V_{1,\Pi}
\]
acts, and the quotient set
\[ H_{c,\Pi} = V_{c,\Pi} / W_\Pi \,. \]
We can extend the map $\delta : C(\Q) \to H_c$ to a map
\[ \delta_\Pi : C(\Q) \To H_{c,\Pi}\,, \quad
                (u_0:v_0:y_0) \longmapsto [\varphi \mapsto \varphi(u_0,v_0), y_0]
\]
where $[\alpha, s]$ denotes the class of $(\alpha, s)$ and
$\varphi(u,v) = \prod_{\theta \in \varphi} (u - \theta v)$. If $y_0 = 0$,
the definition has to be changed suitably, see the definition of~$\delta$ above.

For the following, we will assume that $Y$ contains a partition~$X$ of~$\Theta$,
and that for every $\varphi \in Y$, there is a partition of~$\Theta$
contained in~$Y$ that has~$\varphi$ as an element.
(If necessary, we can extend~$Y$ by adding the complements of its elements;
this does not change the covering group~$G_Y$ or the fields that occur
as components of~$L_Y$.) Then, using the obvious projections $\pi_Y : L_\Pi \to L_Y$,
we can define a group $U_Y = \pi_Y(U_\Pi)$ with a map $\iota_Y : \Q\mult \to U_Y$,
a set
\[ V_{c,Y} = \{(\alpha, s) \in U_Y \times \Q\mult
                  : c \prod_{\varphi \in X} \alpha(\varphi) = s^2\}
             \subset U_Y \times \Q\mult \,,
\]
a group
\[ W_Y = \{(\gamma^2 \iota_Y(a), a^d \prod_{\varphi \in X} \gamma(\varphi))
                       : \gamma \in U_Y, a \in \Q\mult\}
         \subset V_{1,Y}
\]
and the quotient set
\[ H_{c,Y} = V_{c,Y} / W_Y \,. \]
We get induced maps, which we denote again by~$\pi_Y$, from the objects
associated to~$\Pi$ to the corresponding objects associated to~$Y$.
We define $\delta_Y = \pi_Y \circ \delta_\Pi : C(\Q) \to H_{c,Y}$.
Using notations $H_{c,Y,p}$, $\rho_{Y,p}$, $\delta_{Y,p}$ in analogy
to $H_{c,p}$ etc., we define the \emph{$Y\!$-Selmer set} of~$C$ to be
\[ \Sel(C, Y) = \{h \in H_{c,Y} : \rho_{Y,p}(h) \in \Img{\delta_{Y,p}}
                    \text{\ for all places~$p$ of~$\Q$}\} \,.
\]
Projecting to the first component, we obtain the \emph{fake $Y\!$-Selmer set}
$\Sel_\fake(C, Y)$. We write $H'_{c,Y} \subset U_Y/(U_Y^2 \iota_Y(\Q\mult))$
for the image of~$H_{c,Y}$ under the map induced by projection of the
product $U_Y \times \Q\mult$ to the first factor, and similarly $\delta'_Y$
for the composition $C(\Q) \stackrel{\delta_Y}{\to} H_{c,Y} \to H'_{c,Y}$.
Using $H'_{c,Y,p}$, $\rho'_{Y,p}$, $\delta'_{Y,p}$ for the local equivalents,
we have
\[ \Sel_\fake(C, Y)
      = \{h \in H'_{c,Y} : \rho'_{Y,p}(h) \in \Img{\delta'_{Y,p}}
                           \text{\ for all places~$p$ of~$\Q$}\} \,.
\]

It is clear that the covering $D_{\alpha,s,Y} \to C$ only depends on
the image of~$(\alpha,s)$ in~$V_{c,Y}$; therefore we will write
$\pi_{\beta,s,Y} : D_{\beta,s,Y} \to C$ instead, where $(\beta,s)$
is the image of~$(\alpha,s)$ in~$V_{c,Y}$.
As usual, we then have the following result.

\begin{thm} \label{thm:genpd}
  We have $\delta_Y\bigl(C(\Q)\bigr) \subset \Sel(C, Y)$, and
  \[ C(\Q) = \bigcup_{[\beta, s] \in \Sel(C, X)}
                    \pi_{\beta,s,Y}\bigl(D_{\beta,s,Y}(\Q)\bigr) \,.
  \]
\end{thm}

The curve $D_{\beta,s}$ is a connected component of the subscheme
of~$\PP_Y \times C$ defined in terms of the coordinates $z_\varphi$
and $(u:v:y)$ by
\[ \exists a \neq 0 : \quad
    \text{$\beta(\varphi) z^2_\varphi = a^{\deg\varphi} \varphi(u,v)$
          for all $\varphi \in Y$ \quad and \quad
          $c \prod_{\varphi \in X} z_\varphi = a^d y$\,.}
\]
(To select the appropriate component, one has to take into account
possible relations between the~$z_\varphi$ --- we can define $D_{\beta,s}$ as
the closure of the set of all $(\PP(z), Q)$ with $z \in \bar{U}_Y$
satisfying
$(\beta z^2, c \prod_{\varphi \in X} z_\varphi)
   = (\bar{\iota}_Y(a), a^d) \bar{\delta}_Y(Q)$
for some $a \in \bar{\Q}\mult$. For this, we extend the objects and maps
to their $\bar{\Q}$-counterparts.)

%======================================================================

\section{Computing Selmer sets} \label{S:pdesc1}

In this section, we explain how a set like $\Sel_\fake(C, Y)$ can
be computed. The first step is to reduce the infinitely many local
conditions to only a finite set of places.

Let $Y/\GalQ$ be the set of Galois-orbits of~$Y$. For each orbit
$O \in Y/\GalQ$, we select a representative $\varphi_O \in O$.
Then $L_Y \cong \prod_O K_O$, where $K_O$ is the field of definition
of~$\varphi_O$ (i.e., the subfield of~$\bar{\Q}$ consisting of elements
fixed by the stabilizer   
of~$\varphi_O$ in~$\GalQ$). Let $p$ be a (finite)
prime of~$\Q$, and let $\beta \in U_Y$ be an element such that there
is $s \in \Q\mult$ with $[\rho_{Y,p}(\beta), s] \in \Img{\delta_{Y,p}}$.
Let $\beta_O \in K_O$ be the $O$-component of~$\beta$, and let
$\fp$ be a place of~$K_O$ above~$p$. Let $\varphi'(u,v) = f/\varphi(u,v)$
be the cofactor of~$\varphi(u,v)$. By assumption, there are
$u_0, v_0, y_0 \in \Q_p$ such that $\varphi(u_0,v_0) = \beta_O$ and
$f(u_0,v_0) = y_0^2$. If $p$ does not divide the leading coefficient~$c$
of~$f$ and $\fp$ does not divide the resultant
$R_O = \Res(\varphi(u,v), \varphi'(u,v))$,
then the valuation $v_{\fp}(\beta_O)$ must be even. We therefore define
$S_O$ to be the (finite) set of places~$\fp$ of~$K_O$ such that $\fp \mid \infty$
or $v_{\fp}(c) \neq 0$ or $v_{\fp}(R_O) \neq 0$. Write $\cS$ for the
family $(S_O)_{O \in Y/\GalQ}$. As usual, if $K$ is a number field
and $S$ is a set of places of~$K$ containing the infinite places, we define
\[ K(S, 2) = \{\alpha (K\mult)^2 \in K\mult/(K\mult)^2
                : v_{\fp}(\alpha) \text{\ is even for all $\fp \notin S$}\} \,.
\]
Then we can define
\[ L_Y(\cS, 2) = \prod_O K_O(S_O, 2) \subset L_Y\mult/(L_Y\mult)^2 \,. \]
From the discussion above, it follows that elements of the fake Selmer set
$\Sel_\fake(C, Y)$ are represented by elements $\beta$ of~$L_Y(\cS, 2)$.
Since the groups $K(S, 2)$ are finite when $S$ is finite, this gives
the Selmer set as a subset of a finite group. We have to determine the
image of this finite group in the quotient group
$K_Y\mult/((K_Y\mult)^2 \iota_Y(\Q\mult))$.

The map $\iota_Y : \Q\mult \to L_Y\mult$ induces a map again denoted~$\iota_Y$
from $\Q\mult/(\Q\mult)^2 \to L_Y\mult/(L_Y\mult)^2$. By standard results
from algebra, we have that the image of~$L_Y(\cS, 2)$
in $K_Y\mult/((K_Y\mult)^2 \iota_Y(\Q\mult))$ is given by
$L_Y(\cS, 2)/(L_Y(\cS, 2) \cap \iota_Y(\Q\mult/(\Q\mult)^2))$. So we have to determine
the intersection $L_Y(\cS, 2) \cap \iota_Y(\Q\mult/(\Q\mult)^2)$.
We write $e_{\fp/p}$ for the ramification index of the extension $K_{O,\fp}/\Q_p$.

\begin{lem}
  Let $T$ be the following set of rational primes $p$:
  \[ T = \{ p :
            \text{for all $O$ such that $\#\varphi_O$ is odd:
                  $\forall \fp \mid p$, either $\fp \in S_O$
                  or $2 \mid e_{\fp/p}$}\} \,.
  \]
  Then
  \[ L_Y(\cS, 2) \cap \Img{\iota_Y} = \iota_Y(\Q(T,2))\,. \]
\end{lem}
\begin{proof}
  First we show that $\iota_Y(\Q(T,2)) \subset L_Y(\cS, 2)$.
  Suppose $a \in \Q(T,2)$ and $\fp$ is a finite place of~$K_O$, where
  $\#\varphi_O$ is odd (in the other case, there is nothing to show).
  Let $p$ be the rational prime below $\fp$. Now
  \begin{equation} \label{eqn:vfp}
    v_\fp(a) = e_{\fp/p} \cdot v_p(a) \,.
  \end{equation}
  Suppose $v_\fp(a)$ is odd. Thus $v_p(a)$ is odd and $e_{\fp/p}$
  is odd. As $a \in \Q(T,2)$ and $v_p(a)$ is odd, we see that
  $p \in T$. By definition of $T$ we deduce that $\fp \in S_O$.
  This shows that $\iota_Y(a) \in L_Y(\cS,2)$ as required.
  
  Now we want to show that $L_Y(\cS,2) \cap \Img{\iota_Y} \subset \iota_Y(\Q(T,2))$.
  Suppose $\alpha \in L_Y(\cS,2)$ is
  also in $\Img{\iota_Y}$. Then there is some $a \in \Q\mult$ and some 
  $\beta \in L_Y\mult$ such that $\alpha_O \beta_O^2 = a^{\#\varphi_O}$
  for all $O \in Y/\GalQ$. 
  We want to show that $a \in \Q(T,2)$. Suppose $p \notin T$.
  Then there is some place $\fp$ of some~$K_O$ with $\#\varphi_O$ odd
  such that $\fp \mid p$,
  $\fp \notin S_O$ and $e_{\fp/p}$ is odd. As $\fp \notin S_O$, we know that 
  $v_\fp(\alpha_O)$ is even. This implies that $v_\fp(a)$ is even.
  By~\eqref{eqn:vfp}, we see that $v_p(a)$ is even.
  As this is true for all $p \notin T$, we have that $a \in \Q(T,2)$ as required.
\end{proof}

We have already remarked that
the group $L_Y(\cS,2)$ is finite; it is also computable.  
Its computation~\cite{SikSm} requires knowledge of the class groups
of the fields~$K_O$ and of a subgroup of the unit group of each~$K_O$
of full rank and odd index.

\begin{lem}\label{lem:kappaH}
  The fake Selmer set $\Sel_\fake(C, Y)$ is contained in the
  intersection~$H'_{c,Y}$ of
  $L_Y(\cS, 2)/\iota_Y(\Q(T, 2))$ with the image of~$H_{c,Y}$ under the
  projection to $L_Y\mult/((L_Y\mult)^2 \iota_Y(\Q\mult))$.
  If $\beta \in L_Y(\cS, 2)$ represents an element of this intersection,
  then the covering curve $D_{\beta,s,Y}$ (for both possible choices of~$s$)
  has good reduction at all odd primes $p$ not dividing the discriminant 
  or the leading coefficient of~$f$.
\end{lem}
\begin{proof}
  The first assertion follows from the preceding discussion.
  
  If $p$ is an odd prime not dividing the discriminant or the leading coefficient
  of~$f$, then $\beta$ can be represented by a tuple $(\beta_O)_O$ such that
  $\beta_O$ is a $\fp$-adic unit for all $\fp \mid p$. The whole construction
  of covering curves can then be carried out over~$\F_p$. In particular, we
  obtain an unramified Galois covering of $C/\F_p$ that is the reduction of
  $D_{\beta,s,Y}$ mod~$p$, which must therefore be smooth.
\end{proof}

Let $\cA$ be any set of (finite or infinite) places of~$\Q$. 
We denote by $\Sel_\fake(C, Y, \cA)$ the subset of~$H'_{c,Y}$ consisting of
elements that satisfy the local conditions for the fake Selmer set at all
places in~$\cA$. Then
\[ \Sel_\fake(C, Y) = \Sel_\fake(C, Y, \{\text{all places of~$\Q$}\})
     \subset \Sel_\fake(C, Y, \cA) \,.
\]
By definition, $\Sel(C, Y)$ maps to $\Sel_\fake(C, Y)$. So
by Theorem~\ref{thm:genpd}, we see that $C(\Q)$ must be empty if
$\Sel_\fake(C, Y, \cA) = \emptyset$.

This definition will be useful, since we will see that we would need
to check the local conditions at very many primes if we want to compute
the fake Selmer set exactly. Using a smaller number of primes can already
give a very useful upper bound, which is much easier to compute.

The next result shows that we only need to consider a finite set of places
when we want to compute a (fake) Selmer set.

\begin{thm}\label{thm:partialdes}
  Let $\cB$ be a set of rational primes containing the following primes~$p$:
  \begin{itemize}
    \item $p = \infty$,
    \item $p < 4 g_D^2$ where $g_D = \#G_Y (g - 1) + 1$ is the genus of the covering
          curves $D_{\beta,s,Y}$.
    \item $p$ dividing the discriminant or the leading coefficient of~$f$.
  \end{itemize}
  Then $\Sel_\fake(C, Y) = \Sel_\fake(C, Y, \cB)$.
\end{thm}

\begin{proof}
  We have to show that for a prime~$p$ outside~$\cB$, the local condition
  is automatically satisfied. So assume that $p \notin \cB$. Then $p \ge 4 g_D^2$
  and $p$ does not divide the discriminant or leading coefficient of~$f$.
  Let $\beta \in L_Y(\cS, 2)$ represent an element of~$H'_{c,Y}$.
  The latter two conditions on~$p$ imply by Lemma~\ref{lem:kappaH} that
  $D_{\beta,s,Y}$ has good reduction at~$p$. The first condition then implies
  by the Hasse-Weil bounds that the reduction of~$D_{\beta,s,Y}$ mod~$p$ has (smooth) 
  $\F_p$-points; then Hensel's Lemma shows that $D_{\beta,s,Y}(\Q_p)$ is non-empty.
  This in turn means that the element represented by~$\beta$ is in the image
  of $\pr_1 \circ \delta_{Y,p}$. The theorem follows.
\end{proof}

One convenient way of obtaining a suitable set~$Y$ is the following.
We fix some number field~$K$ and let $X$ be the partition of~$\Theta$
corresponding to the factorization of~$f$ into irreducible factors over~$K$.
Then $Y$ can be taken to be the union of the $\GalQ$-orbits of the elements
of~$X$. In this case, all the fields occurring as components of the
algebra~$L_Y$ will be (isomorphic to) subfields of~$K$. We will denote
the corresponding (fake) Selmer sets also by $\Sel(C, K)$, $\Sel(C, K, \cA)$,
$\Sel_\fake(C, K)$ and $\Sel_\fake(C, K, \cA)$.

In the following, we will assume that we are in this situation and would
like to compute $\Sel_\fake(C, K, \cA)$ for some finite set~$\cA$ of primes.
For simplicity, we will assume in addition that none of the factors
in the factorization of~$f$ over~$K$ is actually defined over a smaller
field and that no two of the factors are in the same $\GalQ$-orbit.
Since it is advantageous for the computation, we will remove  
the requirement
that the factors are monic and instead consider a factorization
\[ f(u,v) = c f_1(u,v) f_2(u,v) \cdots f_r(u,v) \]
with $c \in \Z$ and polynomials $f_1, \ldots, f_r$ with coefficients in~$\OO$,
the ring of integers of~$K$. (It may be necessary to scale $f$ by an
integral square to make this possible.)
Then $L_Y \cong K^r$, and if we let $S$ be the union of the sets~$S_j$
corresponding to the orbit of~$f_j$, then $L_Y(\cS, 2) \subset K(S, 2)^r$.

\subsection*{Testing the local conditions}

Let $h \in H'_{c,Y}$, and let $p$ be a rational prime. To be able
to compute $\Sel_\fake(C, K, \cA)$, we need an algorithm for determining
whether $\rho_p(h) \in \Img(\delta'_{Y,p})$ (for all $p \in \cA$). 
Let us deal first with the case $p = \infty$ which is certainly easier.
In this case we actually compute the image of  
$\delta'_{Y,\infty} : C(\mathbb{R}) \rightarrow H'_{c,Y,\infty}$. There are at most
two real points at infinity, and their image is the same as that of nearby
points, so there is no need to consider them separately.
Thus we may restrict to the affine patch given by
$y^2 = f(u,1)$. As we move along the affine patch, the image changes
only where $y = 0$.  Let $I_1, \dots, I_k$ be the open intervals on which
$f(u,1)$ is positive. For each $j$ choose $u_j \in I_j$
and $y_j$ such that $y_i^2 = f(u_j,1)$. Then $\Img(\delta'_{Y,\infty})$
is simply $\{\delta'_{Y,\infty}(u_j:1:y_j) : 1 \le j \le k\}$.

We now let $p$ be a finite prime.
We shall suppose that $h$ is 
represented by $(\alpha_1,\dots,\alpha_r)$ in $K(S,2)^r$. 
We denote the degree of~$f_j$ by~$d_j$ and shall also suppose that
$d_1, \dots, d_s$ are odd and $d_{s+1}, \dots, d_r$ are even.
Then $\rho_p(h) \in \Img(\delta'_{Y,p})$ if and only if there is
some $(u:v) \in \PP^1(\Q_p)$ and $a \in \Q_p\mult/(\Q_p\mult)^2$
such that
$a \alpha_i F_i(u,v) \in K_p^2$ for $1 \leq i \leq s$
and $\alpha_i F_i(u,v) \in K_p^2$ for $s+1 \leq i \leq r$
and $f(u,v) \in \Q_p^2$.
Now $\Q_p\mult/(\Q_p\mult)^2$ is finite, and we choose
rational integers $a$ representing its cosets.
Moreover $(u:v) = (x:1)$
or $(1:px)$ for some $x \in \Z_p$. Thus we can decide whether $h$
maps into the local image st~$p$ if we can decide the following
question: given a polynomial $f \in \Z[x]$ and polynomials
$f_i \in \OO[x]$, is there $x \in \Z_p$ such that 
$f(x) \in \Q_p^2$ and $f_i(x) \in K_p^2$ for $1 \leq i \leq r$. 
Equivalently, is there $x \in \Z_p$ satisfying  
the following property
\begin{equation}\label{eqn:prop}
  \begin{cases}
    \text{$f(x)$ is a square in $\Q_p$ and}\\ 
    \text{$f_i(x)$ is a square in $K_\fp$} & 
      \text{for each place $\fp$ above $p$ and for $1 \leq i \leq r$}.
  \end{cases}
\end{equation}
We shall restrict to the case where 
\begin{equation}\label{eqn:fprop}
  \text{$f$ is separable} \qquad
  \text{and} \qquad
  \text{$f/\prod f_i$ is a constant in $(K\mult)^2$}. 
\end{equation}
This is certainly true for the $f$ and $f_i$ in our situation.

We shall need the following pair of lemmas.
\begin{lem}\label{lem:crit1}
  Let $g=\sum a_i x^i \in \Z[x]$ and $x_0 \in \Z$ and $n \geq 1$.
  Let $c=\min{v_p(a_i)}$ and $m=v_p(g^\prime(x_0))$.
  Write $\lambda=\min \{ m+n , c+2n\}$. 
  Let $k=v_p(g(x_0))$. Suppose $k< \lambda$.
  If either of the following two conditions holds,
  \begin{itemize}
    \item $k$ is odd, or
    \item $k$ is even and $g(x_0)/p^{k}$ is not a square modulo $p^{\lambda-k}$,
  \end{itemize}
  then $g(x)$ is not a square in $\Q_p$ for all $x \in x_0+p^n \Z_p$.
\end{lem}

\begin{lem}\label{lem:crit2}
  Let $g=\sum a_i x^i \in \OO[x]$ and $x_0 \in \Z$ and $n \geq 1$.
  Let $\fp$ be a place above $p$ having ramification index $e$
  and let $\pi \in \OO$ be a uniformizing element for $\fp$.
  Let $c=\min{v_\fp(a_i)}$ and $m=v_\fp(g^\prime(x_0))$.
  Write $\lambda=\min \{ m+n e, c+2ne\}$. 
  Let $k=v_\fp(g(x_0))$. Suppose $k< \lambda$.
  If either of the following two conditions holds,
  \begin{itemize}
    \item $k$ is odd, or
    \item $k$ is even and $g(x_0)/\pi^{k}$ is not a square modulo $\pi^{\lambda-k}$,
  \end{itemize}
  then 
  $g(x)$ is not a square in $K_\fp$ for all $x \in x_0+p^n \Z_p$.
\end{lem}

We shall prove Lemma~\ref{lem:crit2}, with the proof of
Lemma~\ref{lem:crit1} being an easy simplification.

\begin{proof}
  Suppose $x \in x_0+p^n \Z_p \subset x_0+\pi^{ne} \OO_\pi$. 
  Let $g_1=g/\pi^c \in \OO_\pi[x]$.
  By Taylor's Theorem
  \[
  g_1(x) = g_1(x_0)+(x-x_0) g_1^\prime(x_0) + O(\pi^{2ne}).
  \]
  Thus 
  \[
  g(x)=g(x_0) + (x-x_0) g^\prime(x_0) + O(\pi^{c+2ne}).
  \]
  It follows that
  \[
  g(x) \equiv g(x_0) \pmod{\pi^\lambda},
  \]
  where $\lambda$ is given in the statement of the lemma. The lemma follows.
\end{proof}

We return to our question: given $f\in \Z[x]$ and $f_i \in \OO[x]$
satisfying~\eqref{eqn:fprop},
is there $x \in \Z_p$ satisfying~\eqref{eqn:prop}? Our algorithm for
answering this question produces a sequence of finite sets
of integers $\cB_0, \cB_1, \cB_2, \dots$ satisfying 
\begin{equation}\label{eqn:cB}
  \{ x \in \Z_p : \text{$x$ satisfies \eqref{eqn:prop}} \}
     \subset \bigcup_{x_0 \in \cB_n} (x_0 + p^n \Z_p).
\end{equation}
We start with $\cB_0=\{0\}$. To produce $\cB_n$ from $\cB_{n-1}$
we initially let
\[
\cB_n=\{x_0+p^n a : \text{$x_0 \in \cB_{n-1}$ and $0 \leq a \leq p-1$} \}.
\]
If any $x_0 \in \cB_n$ satisfies \eqref{eqn:prop} then
we have answered our question positively and our algorithm
terminates. Otherwise,  
for each $x_0 \in \cB_n$,
we apply the tests in 
Lemmas \ref{lem:crit1} and~\ref{lem:crit2}. 
If the hypotheses of Lemma~\ref{lem:crit1} apply to $g=f$
with the current choices of $x_0$ and $n$ then
we eliminate $x_0$ from $\cB_n$. Likewise if
there is some $1 \leq i \leq r$,
and some $\fp$ above $p$ such that $g=f_i$ satisfies the hypotheses 
of Lemma~\ref{lem:crit2}. 
We see that once this process is complete, $\cB_n$ still
satisfies~\eqref{eqn:cB}. At any stage of the algorithm, if $\cB_n$
is empty then we have answered our question negatively.

\begin{lem}
  The above algorithm terminates in finite time.
\end{lem}

\begin{proof}
  Suppose otherwise. 
  
  Suppose first that there is some $x^* \in \Z_p$ that satisfies
  property \eqref{eqn:prop}
  and $f(x^*) \neq 0$. By \eqref{eqn:fprop} this
  forces $f_i(x^*) \neq 0$. Now by \eqref{eqn:cB} there is, for
  each $n$, an $x_n \in \cB_n$ such that $x_n \equiv x^* \pmod{p^n}$.
  However, for $n$ large enough it is clear that $x_n$ satisfies 
  \eqref{eqn:prop} and so the algorithm would have stopped at the
  $n$-th step, and we have a contradiction.
  
  Next we shall suppose that $x^*$ satisfies property \eqref{eqn:prop}
  and $f(x^*)=0$. Hence precisely one of the $f_i(x^*)$ is zero.
  Without loss of generality suppose that $f_1(x^*)=0$. 
  Now $f^\prime(x^*) \neq 0$ as $f$ is separable. Let 
  $x^{**} =x^*+p^{2u} f^\prime(x^*)$ where $u$ is a large positive
  integer that will be chosen
  later. By Taylor's Theorem 
  \[
  f(x^{**})\equiv \left( p^u f^\prime(x^*) \right)^2 \pmod{p^{4u}}.
  \]
  This forces $f(x^{**})$ to be a non-zero square for $u$ large enough.
  Moreover, for $u$ large enough, $f_i(x^{**})$ is a non-zero square
  in $K_\fp$ for $2 \leq i \leq r$, since $f_i(x^*)$ is
  a non-zero square. By \eqref{eqn:fprop}, $f_1(x^{**})$ must also
  be a square. Thus $f(x^{**}) \neq 0$ and $x^{**}$
  satisfies \eqref{eqn:prop}. This reduces us to the previous
  case and we have a contradiction.
  
  We deduce that no $x^* \in \Z_p$ satisfies \eqref{eqn:prop}.
  Now choose $x_n \in \cB_n$
  for $n=0,1,2,\dots$ such that $x_{n+1} \equiv x_{n} \pmod{p^n}$.
  Let $x^*=\lim x_n \in \Z_p$. In particular $x^* \equiv x_n \pmod {p^n}$.
  Now either $f(x^*)$ is a non-square in $\Q_p$, 
  or $f_i(x^*)$ is a non-square in $K_\fp$ for some 
  $\fp$ above $p$ and some $1 \leq i \leq r$.
  
  Suppose that $f(x^*)$ is a non-square in $\Q_p$.
  Let $k=v_p(f(x^*)) < \infty$. Suppose $n>k$. Then $k=v_p(f(x_n))$. 
  Now either $k$ is odd, or $f(x_n)/p^k$ is not a square
  modulo $p^{n-k}$, for large enough $n$. In either case,
  $x_n$ satisfies the hypotheses of Lemma~\ref{lem:crit1},
  and cannot belong to $\cB_n$ giving a contradiction.
  Likewise we obtain
  a contradiction if $f_i(x^*)$ is a non-square in $K_\fp$.
\end{proof}

%======================================================================

\section{Edwards' Parametrization}

The remainder of this paper is devoted to the proof of Theorem~\ref{thm:main}.
In this section, we use Edwards' parametrization of the generalized
Fermat equation with signature $(2,3,5)$.
This allows us to reduce the resolution of \eqref{eqn:main}
to the determination of rational points on $49$ hyperelliptic curves
of genus $14$; in determining the rational points on these
curves our partial descent will play a major r\^{o}le.
All our computations are carried out using the package
{\tt MAGMA} \cite{MAGMA}.

In the following, the notation $h=[\alpha_0,\alpha_1,\dotsc,\alpha_{12}]$ 
means 
that $h$ is the binary form 
\[ h(u,v)=\sum_{i=0}^{12} \binom{12}{i}\alpha_i u^{i}v^{12-i}. \]
We define binary forms $h_1, \ldots, h_{27}$ as given in Table~\ref{table}
on page~\pageref{table}.

\begin{table}
\begin{center}
\noindent\hrulefill

\smallskip

\begin{tabular}{l}
$h_1=[0,1,0,0,0,0,-144/7,0,0,0,0,-20736,0]$,
\\
$h_2=[-1,0,0,-2,0,0,80/7,0,0,640,0,0,-102400]$,
\\
$h_3=[-1,0,-1,0,3,0,45/7,0,135,0,-2025,0,-91125]$,
\\
$h_4=[1,0,-1,0,-3,0,45/7,0,-135,0,-2025,0,91125]$,
\\
$h_5=[-1,1,1,1,-1,5,-25/7,-35,-65,-215,1025,-7975,-57025]$,
\\
$h_6=[3,1,-2,0,-4,-4,24/7,16,-80,-48,-928,-2176,27072]$,
\\
$h_7=[-10,1,4,7,2,5,80/7,-5,-50,-215,-100,-625,-10150]$,
\\
$h_8=[-19,-5,-8,-2,8,8,80/7,16,64,64,-256,-640,-5632]$,
\\
$h_9=[-7,-22,-13,-6,-3,-6,-207/7,-54,-63,-54,27,1242,4293]$,
\\
$h_{10}=[-25,0,0,-10,0,0,80/7,0,0,128,0,0,-4096]$,
\\
$h_{11}=[6,-31,-32,-24,-16,-8,-144/7,-64,-128,-192,-256,256,3072]$,
\\
$h_{12}=[-64,-32,-32,-32,-16,8,248/7,64,124,262,374,122,-2353]$,
\\
$h_{13}=[-64,-64,-32,-16,-16,-32,-424/7,-76,-68,-28,134,859,2207]$,
\\
$h_{14}=[-25,-50,-25,-10,-5,-10,-235/7,-50,-49,-34,31,614,1763]$,
\\
$h_{15}=[55,29,-7,-3,-9,-15,-81/7,9,-9,-27,-135,-459,567]$,
\\
$h_{16}=[-81,-27,-27,-27,-9,9,171/7,33,63,141,149,-67,-1657]$,
\\
$h_{17}=[-125,0,-25,0,15,0,45/7,0,27,0,-81,0,-729]$,
\\
$h_{18}=[125,0,-25,0,-15,0,45/7,0,-27,0,-81,0,729]$,
\\
$h_{19}=[-162,-27,0,27,18,9,108/7,15,6,-51,-88,-93,-710]$,
\\
$h_{20}=[0,81,0,0,0,0,-144/7,0,0,0,0,-256,0]$,
\\
$h_{21}=[-185,-12,31,44,27,20,157/7,12,-17,-76,-105,-148,-701]$,
\\
$h_{22}=[100,125,50,15,0,-15,-270/7,-45,-36,-27,-54,-297,-648]$,
\\
$h_{23}=[192,32,-32,0,-16,-8,24/7,8,-20,-6,-58,-68,423]$,
\\
$h_{24}=[-395,-153,-92,-26,24,40,304/7,48,64,64,0,-128,-512]$,
\\
$h_{25}=[-537,-205,-133,-123,-89,-41,45/7,41,71,123,187,205,-57]$,
\\
$h_{26}=[359,141,-1,-21,-33,-39,-207/7,-9,-9,-27,-81,-189,-81]$,
\\
$h_{27}=[295,-17,-55,-25,-25,-5,31/7,-5,-25,-25,-55,-17,295]$.
\\
\end{tabular}

\smallskip

\noindent\hrulefill

\medskip

\caption{Definition of the forms $h_i$, $1 \le i \le 27$. \label{table}}
\end{center}
\end{table}

For $i=1,\dots,27$, let
\[
 g_i=\frac{1}{132^2}
\left(
 \frac{\partial^2{h_i}}{\partial{u}^2} 
 \frac{\partial^2{h_i}}{\partial{v}^2}	- 
 \frac{\partial^2{h_i}}{\partial{u}\partial{v}} 
	\frac{\partial^2{h_i}}{\partial{u}\partial{v}} 
\right),
\qquad
f_i=\frac{1}{240}
\left(
 \frac{\partial{h_i}}{\partial{u}} 
  \frac{\partial{g_i}}{\partial{v}}-	
  \frac{\partial{h_i}}{\partial{v}} 
 \frac{\partial{g_i}}{\partial{u}} 
\right).
\]
Let
\begin{gather*}
(f_{i},g_{i},h_{i})=(-f_{i-27},g_{i-27},h_{i-27}),
\qquad i=28,29,\\
(f_{i},g_{i},h_{i})=(-f_{i-25},g_{i-25},h_{i-25}),
\qquad i=30,\dots,41,\\
(f_{i},g_{i},h_{i})=(-f_{i-23},g_{i-23},h_{i-23}),
\qquad i=42,\dots,49.\\
\end{gather*}
Note that the $f_i$, $g_i$ and $h_i$ are binary
forms with integral coefficients,
of degrees $30$, $20$ and $12$ respectively.

\begin{thm} (Edwards \cite{Ed})
Suppose $a$, $b$, $c$ are coprime rational integers satisfying
$a^2+b^3+c^5=0$. Then for some $i=1,\dotsc,49$, there
is a pair of coprime rational integers $u$, $v$ such that
\[
a=f_i(u,v), \qquad b=g_i(u,v), \qquad c=h_i(u,v).
\]
\end{thm}
\begin{proof}
See pages 235--236 of \cite{Ed}, particularly the last paragraph on
page 236.
\end{proof}

%======================================================================

\section{Local Solubility}

\begin{lem}
Suppose $x$, $y$, $z$ are coprime integers satisfying 
equation~\eqref{eqn:main}. Then, for some $i$ in
\begin{equation}\label{eqn:I}
I=\{ 2, 3, 5, 6, 15, 16, 17, 23, 24, 27, 28, 29, 31, 
32, 36, 37, 38, 40, 41, 43, 
44, 47, 49
\}
\end{equation}
there is a pair of coprime integers $u$, $v$, such that
\begin{equation}\label{eqn:param}
y^2=f_i(u,v), \qquad x=g_i(u,v), \qquad z=h_i(u,v).
\end{equation}
\end{lem}
\begin{proof}
From Edwards' Theorem, the conclusion certainly holds for
some $1\leq i \leq 49$. It turns out that for all $i$
the form $f_i$ is square-free and so the equation
$y^2=f_i(u,v)$ defines a hyperelliptic curve 
\[
C_i : y^2=f_i(u,v)
\]
in weighted projective space, where we give $u$, $v$ and $y$
the weights $1$, $1$ and $15$. This hyperellipic curve has
genus $g=14$ since the binary form $f_i$ has degree $30$. 
We tested each $C_i$ for everywhere local solubility using
our implementation of the
algorithm in \cite{MSS}. By the Hasse-Weil bounds,
it is only necessary to test for local solubility at
$\infty$, the primes dividing the discriminant of $f_i$,
and those $<4g^2$. We find that for $i$ in the set
\[
\{1,4,9,10,11,13,14,18,25,26,33,35,39,45,46,48\}
\]
the curve $C_i$ has no $2$-adic points and for
$i=20$, $42$, it has no $3$-adic points. 

We can eliminate
a further eight indicies $i$ as follows. 
Let $S=\{\overline{w}^2 : \overline{w} \in \Z/{256 \Z}\}$ and
$T=\{2\overline{w}: \overline{w} \in \Z/{256 \Z}\}$. Let
\[
U_i=\{(\overline{u},\overline{v}) \in (\Z/{256 \Z})^2 : 
f_i(\overline{u},\overline{v}) \in S, \quad (f_i(\overline{u},
\overline{v}),g_i(\overline{u},\overline{v}),h_i(\overline{u},
\overline{v})) \notin T^3\}.
\] 
If $U_i=\emptyset$ then for any pair of integers $u$, $v$,
if $f_i(u,v)$ is a square then the integers $f_i(u,v)$,
$g_i(u,v)$, $h_i(u,v)$ must all be even and so cannot be
coprime. It turns out that $U_i=\emptyset$ for
$i=7,8,12,19,21,22,30,34$, and so we can eliminate these
indicies from consideration. This leaves us with the set $I$
in the statement of the theorem. 

Our attempts to eliminate other indicies using the corresponding
trick with other prime powers were unsuccessful.
\end{proof}

\bigskip

\noindent {\bf Remark.}
In what follows we will determine the rational points on the
curves $C_i$ for the $23$ values of $i \in I$. There are various
relations between these $23$ curves which
are helpful to bear in mind, even though we shall not use them explicitly.
First, the curves
$C_3$, $C_{17}$ and $C_{47}$ are isomorphic. Secondly, if we
write $i \sim j$ to mean that $C_i$ is a quadratic twist
of $C_j$ then we have
\begin{gather*}
5 \sim 31 \sim 49,
\qquad
6 \sim 32,
\qquad
15 \sim 16, \\
23 \sim 24, \qquad
27 \sim 28 \sim 37 \sim 38,
\qquad
43 \sim 44.
\end{gather*}

%======================================================================

\section{Factorization Types}
Let $G \in \Q[u,v]$ be a binary form, and let $K$ be
a number field. We say $G$ has {\em factorization type} 
$[d_1,d_2,\dotsc,d_n]$  over~$K$ if it factors as a product 
$G=G_1 G_2 \dots G_n$ where $G_j \in K[u,v]$ is irreducible over~$K$
 of degree~$d_j$. The following table records the 
factorization types of~$f_i$ over~$\Q$ for the $23$ values of~$i \in I$.

\begin{center}
\begin{tabular}{|c|c|}
\hline
factorization type of $f_i$ over $\Q$ & $i \in I$ \\
\hline\hline
$[30]$ &  $15, 16, 23, 24, 27, 28, 29, 37, 38, 40, 41, 43, 44$\\
\hline
$[ 10, 20 ]$ & $2, 36$ \\
\hline
$[ 6, 12, 12 ]$ & $3, 17, 47$\\
\hline
$[ 1, 1, 4, 4, 4, 8, 8 ]$ & $5, 6, 31, 32, 49$ \\
\hline
\end{tabular}
\end{center}

\vskip 5mm

We implemented our partial Descent in {\tt MAGMA}, and used
it to deal with the remaining $f_i$ as we now explain.

\subsection{Dealing with factorization type $[30]$}

Here the $f_i$ are irreducible and it is impractical
to compute the class group and units of the degree $30$
number fields  
$\Q[x]/f_i$. It follows from Edwards' construction that the
Galois group of the splitting field of any of the $f_i$ (or $g_i$
or $h_i$) must be isomorphic to a subgroup of $\GL_2(\F_5)/\{\pm I\}$. 
In these $13$
cases where $f_i$ is irreducible, it turns out that the
Galois group is isomorphic to $\GL_2(\F_5)/\{\pm I\}$,
which has order $240$. Now
$\GL_2(\F_5)/\{\pm I\}$ has a subgroup of order $48$
and hence index $5$. By the Galois correspondence, the splitting
field of $f_i$ must contain a subfield $K_i$ of degree $5$. 
It is possible to determine for these fields $K_i$
the class group and unit information needed for the
Selmer set computation. 
It turns out that $\Sel_\fake(C_i,K_i,\cA_i)=\emptyset$
where $\cA_i$ is the primes $<100$, infinity
and the primes dividing the leading coefficient of $f_i$. This shows that
$C_i(\Q)=\emptyset$ for $i=15, 16, 23, 24, 27, 28, 29, 
37, 38, 40, 41, 43, 44$. We briefly indicate in this table
the choice of $K_i$. In all these cases, $f_i$ has factorization
type $[6,24]$ over $K_i$. (It can be checked that the Galois group
of the coverings one would obtain is isomorphic to~$\mu_2^4$, so
we have $g_D = 16 (g - 1) + 1 = 209$.)

\vskip 5mm

\begin{center}
\begin{tabular}{|c|c|}
\hline
$i \in I$  & Defining polynomial for $K_i$ \\
\hline\hline
$15, 16, 23, 24$ &  $x^5 - 10 x^2 - 15 x - 6$  \\
\hline
$27,28,37,38,43,44$ & $x^5 + 20x^2 + 30x + 60$  \\
\hline
$29$ & $ x^5 + 30x^2 + 45x + 18$ \\
\hline
$40$ & $x^5 + 20x^2 + 30x + 6$ \\
\hline
$41$ & $x^5 + 30x^3 + 60x^2 + 45x + 12$ \\
\hline\hline
\end{tabular}
\end{center}

\vskip 5mm

For the remaining factorization types
 we let $K=\Q$ and computed the Selmer  
set $\Sel_\fake(C_i,\Q,\cA_i)$  
where  again
$\cA_i$ is the primes $<100$, infinity
and the primes dividing the leading coefficient of $f_i$. 
In all these cases the Selmer set is non-empty.
However, it turns out that each  unramified cover $D_h$ 
corresponding to an element 
$h$ of these Selmer sets has at least one quotient 
$D_h \rightarrow D^\prime/\Q$,
such that:
\begin{enumerate}
\item[(i)] $D^\prime$ is a curve of genus $1$, and its Jacobian
has rank $0$, or
\item[(ii)] $D^\prime$ is a curve of genus $2$, and its Jacobian
has rank at most $1$.
\end{enumerate}
In either case we have been able determine $D^\prime(\Q)$
(where for (ii) we use Chabauty's method \cite{Flynn},
\cite{MP},\cite{We}). This allows us to determine
the rational points on the $D_h$, and hence on the $C_i$.
We give some details below.

\subsection{Dealing with factorization type $[10,20]$}

Here $i=2$ or $36$. We explain the details for $i=2$; those
for $i=36$ are practically identical. 
Here $f_{2}=F_1 F_2$ where
\begin{align*}
F_1 &= 20736 u^{10} + v^{10} \\
F_2 & = 429981696 u^{20} + 1558683648 u^{15} v^5 - 207484416 u^{10} v^{10} - 75168 u^5 v^{15} + v^{20} . \\
\end{align*}
The Selmer set is 
\[
\Sel_\fake(C_2,\Q,\cA_2)=\left\{ \left( 1\cdot (\Q\mult)^2,1 \cdot (\Q\mult)^2\right) 
\right\}. 
\]
Thus if $(u:v:y) \in C_2(\Q)$ then $F_1(u,v)$ and $F_2(u,v)$ are both
squares. In other words, every rational point $(u:v:y) \in C_2(\Q)$
lifts to a rational point $(u:v:y_1:y_2)$ on the curve
\[
D : \left\{
\begin{array}{l}
F_1(u,v) =  y_1^2,\\
F_2(u,v) =  y_2^2,
\end{array}
\right.
\]
via the map 
\[
\phi : D \rightarrow C_{2},
\qquad (u:v:y_1:y_2) \mapsto (u:v: y_1 y_2).
\]
However, the curve
$D$ covers the genus $2$ curve (given here in affine coordinates)
\[
% D_1^\prime : Y^2=20736 X^5+1, \qquad 
D^\prime : Y^2=X^5+20736,
\]
via 
\[
%(u: v : y_1 : y_2) \mapsto (X,Y)=\left(\frac{u}{v},\frac{y_1}{v^5}\right),
%\qquad
\psi : (u: v : y_1 : y_2) \mapsto (X,Y)=\left(\frac{v}{u},\frac{y_1}{u^5}\right).
\]
To determine $D(\Q)$ and hence $C_{2}(\Q)$ it is enough
to determine %either of $D_1^\prime(\Q)$ or 
$D_2^\prime(\Q)$. Write $J$ for the Jacobian of $D^\prime$.
Using the in-built {\tt MAGMA} routines for descent 
on Jacobians of genus $2$ curves (based on \cite{Stoll}) we were able to
show that 
$J(\Q) \cong \Z/5\Z$. From this one can easily conclude
that $D^\prime(\Q)=\{\infty,(0,\pm 144)\}$. Thus 
\begin{align*}
D(\Q) & =\{(0:1: \pm 1: \pm 1), (1:0: \pm 144: \pm 20736)\}, \\
C_2(\Q) & =\{(0:1:\pm 1), (1:0:\pm 2985984)\}.
\end{align*}
From \eqref{eqn:param} we obtain the following solutions to \eqref{eqn:main}
\[
(x,y,z)=(-1,\pm 1,0), \qquad (-429981696,\pm 2985984,0).
\]
We can exclude the latter pair since we are only interested in
solutions where $x$, $y$, $z$ are coprime.

In this case, we obtain double covers of genus $g_D = 2 \cdot 13 + 1 = 27$,
so $4 g_D^2 = 2916$, and the exact computation of the fake Selmer sets
would be feasible.

\subsection{Dealing with factorization type $[6,12,12]$}

Here $i=3$, $17$ or $47$. The Selmer sets for all have size $1$.
It is unsurprising that all three have the same size
Selmer set since, as we have observed before,
the curves $C_3$, $C_{17}$ and $C_{47}$ are isomorphic.
We give the details for $i=3$ here; the other cases are almost
identical. We can write $f_3=F_1 F_2 F_3$ where
\begin{align*}    
F_1 &=320u^6 + v^6,\\
F_2 &=102400u^{12} + 32000 u^9 v^3 + 16440 u^6 v^6 - 100 u^3 v^9 + v^{12},\\
F_3 &=102400 u^{12} + 896000 u^9 v^3 - 140160 u^6 v^6 - 2800 u^3 v^9 + v^{12}. 
\end{align*}
The Selmer set is 
\[
\Sel_\fake(C_3,\Q,\cA_3)=\left\{  
\left( 1\cdot (\Q\mult)^2,1 \cdot (\Q\mult)^2, 1 \cdot (\Q\mult)^2 \right) 
\right\}. 
\]
As before, every rational point $(u:v:y) \in C_3(\Q)$
lifts to a rational point $(u:v:y_1:y_2:y_3)$ on the curve
\[
D : \left\{
\begin{array}{l}
F_1(u,v) =  y_1^2,\\
F_2(u,v) =  y_2^2,\\
F_3(u,v) = y_3^2,\\
\end{array}
\right.
\]
via the map 
\[
\phi : D \rightarrow C_{3},
\qquad (u:v:y_1:y_2:y_3) \mapsto (u:v: y_1 y_2 y_3).
\]
However, the curve
$D$ covers the elliptic curve 
\[
E : Y^2=X^3+25
\]
via 
\[
\psi : D \rightarrow E, \qquad 
(u:v:y_1:y_2:y_3) \mapsto 
\left(
\frac{20 u^2}{v^2},\frac{5y_1}{v^3}
 \right).
\]
The curve $E$ has rank $0$ and the Mordell--Weil group is
\[
E(\Q)=\{ O, (0,5), (0,-5) \}.
\]
We deduce that the only rational points on $C_3$ are
$(u:v:y)=(0:1:\pm 1)$. These give the solution $(0,1,-1)$
to equation \eqref{eqn:main}.

\subsection{Dealing with factorization type $[1,1,4,4,4,8,8]$}

Here $i=5,6,31,32,49$. In all these cases the Selmer set has exactly
two elements. We give the details for $i=5$; the other cases
are similar. Now $f_5=2 F_1 F_2 F_3 F_4 F_5 F_6 F_7$ where 
\begin{align*}
    F_1 & =v, \qquad F_2 =u, \qquad F_3  = 45 u^4 - v^4, \\
	F_4 &= 405 u^4 + 30 u^2 v^2 + v^4, \qquad
    F_5  =15 u^4 + 10 u^2 v^2 + 3 v^4,\\
    F_6 &= 405 u^8 - 540 u^6 v^2 + 846 u^4 v^4 - 60 u^2 v^6 + 5 v^8,\\
    F_7 & = 50625 u^8 - 13500 u^6 v^2 + 4230 u^4 v^4 - 60 u^2 v^6 + v^8. 
\end{align*}
The Selmer set has representatives 
       $( 3, 2, 5, 5, 15, 5, 1)$ and
        $(5, -6, -1, 1, 3, 5, 1)$.
If $(u:v:y)$ is a rational point on $C_5$ mapping to the first element
of the Selmer set 
then there are rational numbers $a$, $y_1,\dots,y_7$, with
$a \neq 0$, such that
\begin{align*}
    F_1 & =3 a y_1^2, \qquad F_2 =2 a y_2^2, \qquad F_3  = 5 y_3^2, \\
	F_4 &= 5 y_4^2, \qquad
    F_5  = 15 y_5^2, \qquad
    F_6 = 5 y_6^2, \qquad
    F_7  = y_7^2. 
\end{align*}
Consider the curve 
\[
D^\prime : F_3(u,v)=5 y_3^2.
\]
The Jacobian of this genus $1$ curve is the elliptic curve
\[
E : y^2 = x^3 + 4500x,
\]
which has rank $0$ and Mordell--Weil group $E(\Q)=\{O,(0,0)\}$.
It follows that $D^\prime(\Q)=\{ (1:0:3) , (1:0:-3) \}$.
This gives us the solution $(184528125,0,-91125)$ to \eqref{eqn:main}
which we can exclude since we are only interested in primitive solutions. 
We deal with the second element of the Selmer set in a similar way.

%======================================================================


\begin{thebibliography}{99}

\bibitem{Ben} M.\ Bennett,
{\em On the equation $x^{2n}+y^{2n}=z^5$},
J.\ Th\'{e}or.\ Nombres Bordeaux  {\bf 18}  (2006),
315--321.

\bibitem{BEN} M.\ A.\ Bennett, J.\ S.\ Ellenberg and N.\ C.\ Ng,
{\em The Diophantine equation $A^4+2^\delta B^2=C^n$},
 International Journal of Number Theory, 
{\bf 6} (2010), no.\ 2, 311–338.  

%\bibitem{BenS} M.\ A.\ Bennett and C.\ M.\ Skinner,
%{\em Ternary Diophantine equations via Galois
%representations and modular forms},
%Canad.\ J.\ Math.\  {\bf 56}  (2004),  no.\ 1, 23--54.

\bibitem{Be} F.\ Beukers,
{\em The Diophantine equation $Ax^p+By^q=Cz^r$},
Lectures held at Institut Henri Poincar\'{e},
September 2004,
{\tt http://www.math.uu.nl/people/beukers/Fermatlectures.pdf}  

\bibitem{MAGMA} W.\ Bosma, J.\ Cannon and C.\ Playoust: {\em The Magma
Algebra System I: The User Language}, J.\ Symb.\ Comp.\ {\bf 24} (1997),
235--265. (See also {\tt http://magma.maths.usyd.edu.au/magma/})



%\bibitem{Br3} N.\ Bruin,
%{\em The primitive solutions to $x\sp 3+y\sp 9=z\sp 2$},
%Journal of  Number Theory {\bf 111} (2005), no.\ 1, 179--189. 

\bibitem{BF} N.\ Bruin and E.\ V.\ Flynn,
{\em Towers of $2$-covers of hyperelliptic curves},
Trans.\ Amer.\ Math.\ Soc.\ {\bf 357} (2005),
4329--4347.

\bibitem{BS} N.\ Bruin and M.\ Stoll,
{\em Deciding existence of rational points on curves: an experiment},
Experimental Mathematics {\bf 17} (2008), 181--189.


\bibitem{BS2} N.\ Bruin and M.\ Stoll,
{\em Two-cover descent on hyperelliptic curves},
Mathematics of Computations {\bf 78} (2009), 
2347--2370.

\bibitem{BS3} N.\ Bruin and M.\ Stoll,
{\em The Mordell-Weil sieve: Proving non-existence of
rational points on curves}, LMS J. Comput. Math. {\bf 13} (2010), 272--306.




%\bibitem{CF} J.\ W.\ S.\ Cassels and E.\ V.\ Flynn, 
%{\em Prolegomena to a middlebrow arithmetic of
%curves of genus $2$}, L.M.S. lecture notes series {\bf 230},
%Cambridge University Press, 1997.

\bibitem{ChenS} I.\ Chen and S.\ Siksek,
{\em Perfect powers expressible as sums of two cubes},
Journal of Algebra {\bf 322} (2009), 638--656.

\bibitem{Cohen} H.\ Cohen,  
{\em Number Theory, Volume II: Analytic and Modern Tools},
GTM {\bf 240}, Springer-Verlag, 2007.

%\bibitem{Dahmen} S.\ R.\ Dahmen,
%{\em Classical and modular methods applied to Diophantine equations},
%University of Utrecht, PhD thesis, 2008.

\bibitem{Da97} H.\ Darmon,
{\em Faltings plus epsilon, Wiles plus epsilon, and the generalized
Fermat equation},
C.\ R.\ Math.\ Rep.\ Acad.\ Sci.\ Canada  {\bf 19}  (1997),  no.\ 1, 3--14.


\bibitem{DG} H.\ Darmon and A.\ Granville,
{\em On the Equation $z^{m} = F(x,y)$ and $Ax^{p}+By^{q}=Cz^{r}$},
Bull.\ London Math.\ Society, {\bf 27} (1995), no.\ 6, 513--543.


\bibitem{DM} H.\ Darmon and L.\ Merel,
{\em Winding quotients
and some variants of Fermat's Last Theorem},
J.\ reine angew.\ Math.\ {\bf 490} (1997), 81--100.


\bibitem{Ed} J.\ Edwards,
{\em A complete solution to $X^2+Y^3+Z^5=0$},
J.\ reine angew.\ Math.\ {\bf 571} (2004), 213--236.

\bibitem{El} J.\ Ellenberg,
{\em Galois representations
attached to $\Q$-curves and the generalized Fermat equation
$A^{4} + B^{2} = C^{p}$},
Amer.\  J.\  Math.\ {\bf 126} (2004), 763--787.




\bibitem{Flynn} E.\ V.\ Flynn,
{\em A flexible method for applying Chabauty's Theorem},
Compositio Math.\ {\bf 105} (1997), 79--94.

%\bibitem{FS} E.\ V.\ Flynn and N.\ P.\ Smart,
%{\em Canonical heights on the Jacobians of curves of 
%genus $2$ and the infinite descent},
%Acta Arith.\ {\bf 79} (1997), no.\ 4, 333--352. 

%\bibitem{FW1} E.\ V.\ Flynn and J.\ L.\ Wetherell,
%{\em Finding rational points on bielliptic genus 2 curves},
%Manuscripta Math.\ {\bf 100} (1999), no.\ 4, 519--533. 

%\bibitem{FW2} E.\ V.\ Flynn and J.\ L.\ Wetherell,
%{\em Covering collections and a challenge problem of Serre},
%Acta Arith.\  {\bf 98}  (2001),  no.\ 2, 197--205.



%\bibitem{Kr} A.\ Kraus,
%{\em Sur l'\'equation $a^{3}+b^{3} = c^{p}$},
%Experimental Math.\ {\bf 7} (1998), 1--13.

\bibitem{Kr99} A.\ Kraus,
{\em On the Equation $x^p+y^q=z^r$: A Survey},
Ramanujan Journal {\bf 3} (1999), 315--333.



%\bibitem{Mc1} W.\ G.\ McCallum,
%{\em The arithmetic of Fermat curves},
%Math.\ Ann.\ {\bf 294} (1992), no.\ 3, 503--511.

%\bibitem{Mc2} W.\ G.\ McCallum,
%{\em On the method of Coleman and Chabauty},
%Math.\ Ann.\ {\bf 299} (1994), no.\ 3, 565--596. 

\bibitem{MP} W.\ McCallum and B.\ Poonen,
{\em The method of Chabauty and Coleman},
preprint, 14 June 2010, 
{\tt http://www-math.mit.edu/~poonen/papers/chabauty.pdf} 

\bibitem{MSS} J.\ R.\ Merriman, S.\ Siksek and N.\ P.\ Smart,
{Explicit $4$-descents on elliptic curves},
Acta Arithmetica {\bf LXXVII.4} (1996), 358-–404.


\bibitem{PS} B.\ Poonen and E.\ F.\ Schaefer,
{\em Explicit descent on cyclic covers of the projective line},
J.\ reine angew.\ Math.\ {\bf 488} (1997), 141--188.

\bibitem{PSS} B.\ Poonen, E.\ F.\ Schaefer and M.\ Stoll,
{\em Twists of $X(7)$ and primitive solutions to 
$x^2+y^3=z^7$},
Duke Math.\ J.\ {\bf 137} (2007), 103--158.


\bibitem{Sch} E.\ F.\ Schaefer,
{\em $2$--descent on the Jacobians of hyperelliptic curves},
J.\ Number Theory {\bf 51} (1995), 219--232.

\bibitem{SikSm} S.\ Siksek and N.\ P.\ Smart,
{\em On the complexity of computing the $2$-Selmer group
of an elliptic curve},
Glasgow Mathematical Journal {\bf 39} (1997), 251--258.

%\bibitem{Si} S.\ Siksek,
%{\em Infinite descent on elliptic curves},
% Rocky Mountain J.\ Math.\  {\bf 25}  (1995),  no. 4, 1501--1538. 

%\bibitem{Sikphd} S.\ Siksek,
%{\em Descents on Curves of Genus $1$},
%Ph.D.\ thesis, University of Exeter, 1995.


%\bibitem{StollI} M.\ Stoll,
%{\em On the arithmetic of the curves $y^2 = x^l + A$ and their Jacobians},
%J.\ reine angew.\ Math.\ {\bf 501} (1998), 171--189.

%\bibitem{Stollh1} M.\ Stoll,
%{\em On the height constant for curves of genus two},
%Acta Arith.\ {\bf 90} (1999), 183--201.

\bibitem{Stoll} M.\ Stoll,
{\em Implementing 2-descent for Jacobians of hyperelliptic curves},
Acta Arith.\ {\bf 98} (2001), 245--277.

%\bibitem{StollII} M.\ Stoll,
%{\em On the arithmetic of the curves $y^2 = x^l + A$, II},
%J. Number Theory {\bf 93} (2002), 183--206.

%\bibitem{Stollh2} M.\ Stoll,
%{\em On the height constant for curves of genus two, II},
%Acta Arith.\ {\bf 104} (2002), 165--182.

%\bibitem{Stollch} M.\ Stoll,
%{\em Independence of rational points on twists
%of a given curve},
%Compositio Math. {\bf 142} (2006), 1201--1214.

%\bibitem{Stoll06} M. Stoll,
%{\em On the number of rational squares at fixed distance from a fifth power},
%Acta Arith.\ {\bf 125} (2006), 79--88.

\bibitem{TW} R.\ Taylor and A.\ Wiles,
{\em Ring-theoretic properties of certain Hecke algebras},
Annals of Mathematics {\bf 141} (1995), no.\ 3, 553--572.


\bibitem{We} J.\ L.\ Wetherell,
{\em Bounding the Number of Rational Points on Certain
Curves of High Rank},
Ph.D.\ dissertation, University of California at Berkeley, 1997.

\bibitem{W} A.\ Wiles,
{\em Modular elliptic curves and Fermat's Last Theorem},
Annals of Mathematics {\bf 141} (1995), no.\ 3, 443--551.
\end{thebibliography}
\end{document}